\documentclass[11pt]{article}

\usepackage[a4paper,margin=1.1in]{geometry}
\usepackage[T1]{fontenc}
\usepackage{lmodern}
\usepackage{amsmath,amssymb,amsthm,mathtools}
\usepackage{enumitem}
\usepackage{array}
\usepackage{booktabs}
\usepackage{microtype}
\usepackage{hyperref}
\hypersetup{
  hidelinks,
  pdftitle={Representation Defects and Cassels Pairings for Congruent Number Curves: Redei Symbols and Governing Fields},
  pdfauthor={Shisong Xu},
  pdfkeywords={congruent numbers, Cassels pairing, 2-Selmer groups, ternary quadratic forms, Redei symbols, governing fields}
}
\usepackage{aliascnt}
\usepackage[nameinlink,capitalise]{cleveref}

\newtheorem{theorem}{Theorem}[section]

\newaliascnt{corollary}{theorem}
\newtheorem{corollary}[corollary]{Corollary}
\aliascntresetthe{corollary}

\newaliascnt{proposition}{theorem}
\newtheorem{proposition}[proposition]{Proposition}
\aliascntresetthe{proposition}

\newaliascnt{lemma}{theorem}
\newtheorem{lemma}[lemma]{Lemma}
\aliascntresetthe{lemma}

\newaliascnt{remark}{theorem}
\newtheorem{remark}[remark]{Remark}
\aliascntresetthe{remark}

\crefname{theorem}{Theorem}{Theorems}
\Crefname{theorem}{Theorem}{Theorems}
\crefname{proposition}{Proposition}{Propositions}
\Crefname{proposition}{Proposition}{Propositions}
\crefname{lemma}{Lemma}{Lemmas}
\Crefname{lemma}{Lemma}{Lemmas}
\crefname{corollary}{Corollary}{Corollaries}
\Crefname{corollary}{Corollary}{Corollaries}
\crefname{remark}{Remark}{Remarks}
\Crefname{remark}{Remark}{Remarks}
\crefname{section}{Section}{Sections}
\Crefname{section}{Section}{Sections}

\DeclareFontFamily{U}{wncy}{}
\DeclareFontShape{U}{wncy}{m}{n}{<->wncyr10}{}
\DeclareFontShape{U}{wncy}{b}{n}{<->wncyb10}{}
\DeclareSymbolFont{cyrletters}{U}{wncy}{m}{n}
\SetSymbolFont{cyrletters}{bold}{U}{wncy}{b}{n}
\DeclareMathSymbol{\Sha}{\mathord}{cyrletters}{"58}

\newcommand{\Q}{\mathbb Q}
\newcommand{\Z}{\mathbb Z}
\newcommand{\F}{\mathbb F}

\newcommand{\Sel}{\mathrm{Sel}}
\newcommand{\Pf}{\mathrm{Pf}}
\newcommand{\rk}{\mathrm{rank}}

\newcommand{\rad}{\mathrm{rad}}
\DeclareMathOperator{\corank}{corank}

\title{Representation Defects and Cassels Pairings for Congruent Number Curves: R\'edei Symbols and Governing Fields}
\author{Shisong Xu\\[0.35em]
\small Department of Mathematics, Nanjing University\\
\small 22 Hankou Road, Nanjing, Jiangsu 210093, People's Republic of China\\
\small Email: \href{mailto:shsxu@smail.nju.edu.cn}{shsxu@smail.nju.edu.cn}}
\date{}

\begin{document}
\maketitle

\begin{abstract}
Using BSD results for CM elliptic curves, we relate Qin's quadratic form
representation defects to the Cassels pairing on the pure $2$-Selmer group
of a congruent number elliptic curve.  When the pure $2$-Selmer dimension
is even, we prove that the normalized representation defect modulo $2$ is
the Pfaffian of the Cassels pairing matrix; the dimension of its radical
yields sharper $2$-adic divisibility and information on the $2$-primary
Shafarevich--Tate group.  For products of primes congruent to $1$ modulo
$8$ that are pairwise quadratic residues, we give explicit Cassels pairing
matrices for both $E_n$ and $E_{2n}$ and express their entries in terms of
quartic and R\'edei symbols.  For each fixed prime $p$, we construct a
governing field of degree $256$ and determine the exact joint distribution
of the two Pfaffians by Chebotarev's theorem.  In particular, there is a set
of primes $q$ of natural density $5/128$ for which both $E_{17q}$ and
$E_{34q}$ have rank zero and $2$-primary Shafarevich--Tate group isomorphic
to $(\Z/2\Z)^4$.
\end{abstract}



\section{Introduction}\label{sec:intro}

The congruent number problem asks which positive integers occur as the area of a
right triangle with rational side lengths.  For a positive squarefree integer
$n$, this classical Diophantine question is equivalent to asking whether the
elliptic curve
\[
E_n:\quad y^2=x^3-n^2x=x(x-n)(x+n)
\]
has positive Mordell--Weil rank over $\Q$.  The family $E_n$ has consequently
served as a particularly concrete testing ground for several different kinds
of arithmetic information: central values of $L$-functions, representation
numbers of quadratic forms, $2$-descent, class groups, and the
Shafarevich--Tate group.  The purpose of this paper is to identify a direct
link between these viewpoints and to show that, in suitable families, the same
invariant also governs prime distributions.

Tunnell's theorem \cite{Tunnell} relates the congruent number problem to
Fourier coefficients of modular forms of weight $3/2$; see also
\cite{Koblitz}.  Qin \cite{Qin22} developed a different quadratic form
realization of this relation.  In the odd case the relevant genus consists of
the two classes represented by
\[
f=x^2+2y^2+32z^2,\qquad g=2x^2+4y^2+9z^2-4yz.
\]
The difference of their theta series is cuspidal, whereas the genus average is
governed by imaginary quadratic class numbers.  Qin thereby related the
central value problem to the representation defect $r_f(n)-r_g(n)$, and his
later formulas express the analytic order of the Shafarevich--Tate group in
terms of the same representation numbers \cite{Qin22,Qin24}.  Thus the
representation defect controls the size predicted by BSD.  A natural further
question is whether it can detect the finer $2$-primary structure of
$\Sha(E_n)$ rather than only its order.  Our first result gives a positive
answer: after a canonical power of $2$ is removed, the parity of the defect is
the Pfaffian of the Cassels pairing.

The descent side of the congruent number problem has developed in parallel.
Ouyang--Zhang used $2$-descent, and then second $2$-descent through several
$2$-isogenies and their duals, to construct families of noncongruent numbers
and to control low dimensional Selmer groups
\cite{OuyangZhang14,OuyangZhang15}.  Wang made the Cassels pairing into an
explicit arithmetic criterion: for broad families of congruent number curves
he characterized rank zero cases with
$\Sha(E_n)[2^\infty]\cong(\Z/2\Z)^2$ and obtained sufficient criteria for
larger elementary abelian $2$-primary groups \cite{Wang16}; he also studied
the distribution of the residue symbols entering these criteria
\cite{Wang17}.  Zhang subsequently compared Cassels pairings for related
elliptic curves \cite{Zhang23}.  More recently, for products built from a
noncongruent factor, he computed pairing values on natural descent classes by
means of primitive solutions of ternary norm equations, and translated
nondegeneracy into conditions involving quadratic class groups and tame
kernels \cite{Zhang26}.

A third ingredient is provided by R\'edei reciprocity.  In Stevenhagen's
formulation, the additive R\'edei symbol is a symmetric trilinear
$\F_2$-valued invariant encoding $8$-rank phenomena in quadratic class
groups; it also admits a governing field interpretation in which variation of
the symbol is read from Frobenius classes \cite{Stevenhagen}.  This language
is particularly well suited to the present problem.  Explicit second descent
formulas are naturally written in terms of auxiliary solutions of norm
equations.  R\'edei symbols replace those choices by intrinsic reciprocity
data, and governing fields then turn the resulting identities into exact
density statements.

The main point of the paper is that these three lines of ideas fit into a
single structure.  A priori, a representation defect is a Fourier
coefficient of a difference of ternary theta series, while the Cassels
pairing is a cohomological object on a Selmer group.  We prove that the two
are linked by a Pfaffian identity.  We then express the entries of the
Cassels matrix in quartic and R\'edei symbols and place the resulting
Pfaffians in an explicit governing field.  Thus the chain
\[
\begin{aligned}
\text{quadratic form representation defects}
&\longleftrightarrow \text{Cassels pairings and Pfaffians}\\
&\longleftrightarrow \text{quartic and R\'edei reciprocity}\\
&\longleftrightarrow \text{Chebotarev distributions}.
\end{aligned}
\]
is made explicit in one family.  To our knowledge, this is the first framework
in which these four layers are connected directly.  Besides giving
noncongruence criteria, the connection recovers information on the elementary
divisors of the $2$-primary Shafarevich--Tate group and produces exact joint
densities for pairs of congruent number curves.

We now state the main results.  Complete $2$-descent gives the pure
$2$-Selmer group
\[
\Sel_2'(E_n):=\Sel_2(E_n)/E_n(\Q)[2],\qquad
s_2(n):=\dim_{\F_2}\Sel_2'(E_n),
\]
with
\[
s_2(n)=\rk E_n(\Q)+\dim_{\F_2}\Sha(E_n)[2].
\]
For the even companion of Qin's construction put
\[
s=x^2+4y^2+32z^2,\qquad
 t=4x^2+4y^2+9z^2-4yz,
\]
and, for each form $F\in\{f,g,s,t\}$, write
$r_F(m)=\#\{(x,y,z)\in\Z^3:F(x,y,z)=m\}$.
All representations, including signs and zero coordinates, are counted.  For
squarefree $n$ define
\begin{equation}\label{eq:defect}
\Delta_Q(n)=
\begin{cases}
 r_f(n)-r_g(n),& n\text{ odd},\\[2mm]
 r_s(n/2)-r_t(n/2),& n\text{ even},
\end{cases}
\qquad
\sigma_Q(n)=
\begin{cases}
 \sigma_0(n),&n\text{ odd},\\
 \sigma_0(n/2),&n\text{ even}.
\end{cases}
\end{equation}
Here $\sigma_0(m)=\sum_{d\mid m}1$.  Qin's formulas give, in analytic rank
zero,
\begin{equation}\label{eq:qin-intro}
 |\Sha(E_n)|_{\mathrm{an}}
 =\left(\frac{\Delta_Q(n)}{2\sigma_Q(n)}\right)^2,
\end{equation}
with the corresponding vanishing of $\Delta_Q(n)$ when $L(E_n,1)=0$; see
\cite[Thms.~3.20, 3.21]{Qin22} and \cite[Cor.~2.1, Cor.~2.2]{Qin24}.
The group $\Sel_2'(E_n)$ carries the alternating Cassels pairing
\cite{Cassels,Cassels98}.

For an alternating matrix $C$ of even size let $\Pf(C)$ denote its Pfaffian,
with $\Pf(\varnothing)=1$, and set $v_2(0)=+\infty$.  When $s_2(n)$ is even,
put
\begin{equation}\label{eq:DQ}
\mathcal D_Q(n)
 :=\frac{\Delta_Q(n)}{2\sigma_Q(n)2^{s_2(n)/2}}\pmod2.
\end{equation}
The following theorem is the basic bridge between representation theory and
Selmer arithmetic.

\begin{theorem}[Divisibility and the Pfaffian]\label{thm:main-principle}
Let $n$ be a positive squarefree integer and suppose that $s_2(n)$ is even.  For
any basis $\mathcal B$ of $\Sel_2'(E_n)$, let $C_{\mathcal B}(E_n)$ be the Gram
matrix of the Cassels pairing.  Then
\begin{equation}\label{eq:valuation-intro}
v_2\!\left(\frac{\Delta_Q(n)}{2\sigma_Q(n)}\right)
 \geq \frac{s_2(n)}2,
\end{equation}
so in particular
\[
2\sigma_Q(n)2^{s_2(n)/2}\mid \Delta_Q(n).
\]
Moreover equality holds in \eqref{eq:valuation-intro} if and only if
\begin{equation}\label{eq:main-principle}
\mathcal D_Q(n)=\Pf\bigl(C_{\mathcal B}(E_n)\bigr)=1.
\end{equation}
Equivalently,
\[
\mathcal D_Q(n)=\Pf\bigl(C_{\mathcal B}(E_n)\bigr)\in\F_2,
\]
and the common value is $1$ precisely when
\[
\rk E_n(\Q)=0,
\qquad
\Sha(E_n)[2^\infty]\cong(\Z/2\Z)^{s_2(n)}.
\]
\end{theorem}

The theorem is unconditional.  Its proof combines Qin's central value
formula with the rank zero BSD theorem for these CM curves due to
Burungale--Flach \cite[Cor.~2]{BurungaleFlach} and the rank zero
$2$-converse of Burungale--Tian \cite{BurungaleTian}.  The Pfaffian identity
is stronger than a nonvanishing criterion: it identifies the precise parity
of the normalized representation defect.  The radical of the pairing gives
a further refinement.  If
$d_2(n)=\dim_{\F_2}\rad C_{\mathcal B}(E_n)$, then
Theorem~\ref{thm:radical-depth} gives
\[
v_2\!\left(\frac{\Delta_Q(n)}{2\sigma_Q(n)}\right)
\geq\frac{s_2(n)+d_2(n)}2,
\]
and equality is equivalent to rank zero together with
\[
\Sha(E_n)[2^\infty]\cong
(\Z/2\Z)^{s_2(n)-d_2(n)}\oplus(\Z/4\Z)^{d_2(n)}.
\]
Thus the depth of the $2$-adic divisibility records information that is not
visible from the Selmer dimension alone.

To make the Pfaffian principle explicit, we next consider the first case in
which the pairing contains genuinely higher dimensional information.  Let
$p\ne q$ satisfy
\begin{equation}\label{eq:layer}
p\equiv q\equiv1\pmod8,
\qquad
\left(\frac qp\right)=1.
\end{equation}
Then the Monsky matrices vanish and
$s_2(pq)=s_2(2pq)=4$.  For an odd positive integer $m$ and $(A,m)=1$ write
\[
\left[\frac{A}{m}\right]\in\F_2,
\qquad
(-1)^{[A/m]}=\left(\frac{A}{m}\right).
\]
For $\ell\equiv1\pmod8$, choose $r_\ell^2\equiv2\pmod\ell$ and put
$c_\ell=[(r_\ell+1)/\ell]$.  This value is independent of the root.  Choose
primitive positive solutions
\[
p\alpha^2+q\beta^2=4\gamma^2,\qquad
p\alpha'^2-q\beta'^2=4\gamma'^2,\qquad
p\alpha''^2-2q\beta''^2=4\gamma''^2,
\]
and set
\begin{equation}\label{eq:symbols-intro}
u=\left[\frac{\gamma}{q}\right],\qquad
v=\left[\frac{\gamma}{p}\right],\qquad
\delta=\left[\frac{\gamma'\gamma''}{q}\right].
\end{equation}
The associated conics have rational points by descent and the Hasse principle.
Define
\begin{equation}\label{eq:Pi-intro}
\Pi(p,q)=c_pc_q+c_pu+c_qv+\delta(u+v)\in\F_2.
\end{equation}

\begin{theorem}[Selmer dimension four]\label{thm:intro-four}
Assume \eqref{eq:layer}.  The value $\Pi(p,q)$ is independent of all auxiliary
choices and is the Pfaffian of the complete $4\times4$ Cassels matrix on
$\Sel_2'(E_{pq})$.  In particular,
\[
\Pi(p,q)=1
\]
if and only if
\[
pq\text{ is noncongruent},
\qquad
\Sha(E_{pq})[2^\infty]\cong(\Z/2\Z)^4.
\]
\end{theorem}

The auxiliary norm equations disappear after passing to reciprocity symbols.
Put
\[
U=\left[\frac pq\right]_4,\quad V=\left[\frac qp\right]_4,
\quad R=[2,p,q]_{\mathrm R},\quad e_\ell=\frac{\ell-1}{8}\pmod2,
\]
where $[\cdot/\ell]_4$ is the additive rational quartic symbol and
$[\cdot,\cdot,\cdot]_{\mathrm R}$ is Stevenhagen's additive R\'edei symbol.
Theorem~\ref{thm:choice-free} proves the identities independent of auxiliary choices
\begin{align*}
\Pi(p,q)&=c_pc_q+c_pU+c_qV+R(U+V),\\
\Pi_2(p,q)&=R+e_pe_q+e_pU+e_qV,
\end{align*}
where $\Pi_2$ is the Pfaffian for $E_{2pq}$.  In particular,
$u=U$, $v=V$, and $\delta=R$.  The R\'edei term is essential: as shown in
Proposition~\ref{prop:quartic-insufficient}, the dyadic data and the two mutual
quartic symbols alone do not determine either Pfaffian.  Theorem~\ref{thm:redei-block}
extends the R\'edei description to both $2k\times2k$ Cassels matrices for a
product of $k$ primes.

The same invariant also has a direct interpretation in terms of quadratic
forms and class numbers.  For $m>0$ let $h(-m)$ denote the class number of
$\Q(\sqrt{-m})$, and put
\[
a=x^2+32y^2+32z^2,
\qquad
r_a(m)=\#\{(x,y,z)\in\Z^3:m=a(x,y,z)\}.
\]
Qin's genus identities lead to the following representation congruence.

\begin{theorem}[Representation congruence modulo $32$]\label{thm:intro-congruence}
Assume \eqref{eq:layer}.  Then
\begin{equation}\label{eq:mod32-intro}
\frac{r_f(pq)-r_g(pq)}{32}
 =\frac{r_a(pq)-h(-pq)}{16}
 \equiv\Pi(p,q)\pmod2.
\end{equation}
Equivalently,
\[
r_a(pq)\equiv h(-pq)+16\Pi(p,q)\pmod{32}.
\]
\end{theorem}

Finally, the reciprocity description makes the Pfaffians amenable to
Chebotarev.  For every fixed prime $p\equiv1\pmod8$ we construct an explicit
governing field of degree $256$ whose Frobenius classes simultaneously
control the five binary invariants entering $\Pi$ and $\Pi_2$.  This gives an
exact joint distribution, rather than only a positive density statement.

\begin{theorem}[Exact joint densities]\label{thm:intro-density}
Fix a prime $p\equiv1\pmod8$.  For $s,t\in\F_2$, the primes $q\ne p$ satisfying
$q\equiv1\pmod8$, $(q/p)=1$, and
$(\Pi(p,q),\Pi_2(p,q))=(s,t)$ have natural density within the primes
\[
\frac{4+(-1)^{s+c_p}
+\mathbf1_{\{c_p=e_p\}}(-1)^{s+t}}{128}.
\]
For $p=17$, these densities are $3/128,3/128,5/128,5/128$, in the order
$(0,0),(0,1),(1,0),(1,1)$.
\end{theorem}

The two Pfaffians are independent in their limiting distribution exactly when
$2$ is not a fourth power modulo $p$.  In particular, for $p=17$ the pair
$(\Pi,\Pi_2)=(1,1)$ occurs on a set of primes of density $5/128$; on this set
both $E_{17q}$ and $E_{34q}$ have rank zero and
$2$-primary Shafarevich--Tate group $(\Z/2\Z)^4$.  The field theoretic input
is the intersection calculation in Lemma~\ref{lem:density-fields}, which
identifies the relevant degree $256$ Galois group and its conjugacy classes.

The preceding two prime formulas are not isolated phenomena.  If
$n=p_1\cdots p_k$ with every $p_i\equiv1\pmod8$ and all pairwise Legendre
symbols equal to $1$, Section~\ref{sec:general} gives the Cassels matrix in
block form
\[
C_n=\begin{pmatrix}X&Y\\Y^T&Z\end{pmatrix}
\]
and expands its Pfaffian over perfect matchings.  In particular,
\[
4^k\mid r_a(n)-h(-n),\qquad
r_a(n)\equiv h(-n)+4^k\Pf(C_n)\pmod{2\cdot4^k}.
\]
This shows that the representation theoretic divisibility and the Pfaffian
criterion persist in arbitrary even pure Selmer dimension under the stated
local hypotheses.

The paper is organized as follows.  Section~\ref{sec:framework} establishes
the general divisibility theorem and the refinement coming from the radical
of the Cassels pairing.  Section~\ref{sec:four} treats the four dimensional
Selmer case $n=pq$ and $2pq$.  Section~\ref{sec:general} gives the Cassels
matrices for products of any number of primes, their R\'edei symbol formulas,
and the governing field and group structure consequences.  Finally,
Appendix~\ref{app:fifteen} applies the general framework beyond the
setting of primes congruent to $1$ modulo $8$ and pairwise quadratic residues,
to the families $15p$ and $30p$.

\section{Representation defects and the Cassels pairing}\label{sec:framework}

We begin with the Monsky matrices.  For an odd squarefree integer $n=p_1\cdots p_k$, let
$A_n=(a_{ij})\in M_k(\F_2)$ be given by
\[
a_{ij}=\left[\frac{p_j}{p_i}\right]\quad(i\ne j),
\qquad
a_{ii}=\left[\frac{n/p_i}{p_i}\right],
\]
and let
\[
D_{n,\varepsilon}
 =\operatorname{diag}\left(\left[\frac{\varepsilon}{p_1}\right],\ldots,
 \left[\frac{\varepsilon}{p_k}\right]\right).
\]
With the standard choice of coordinates, Monsky's matrices are
\begin{equation}\label{eq:Monsky}
M_n=
\begin{pmatrix}
A_n+D_{n,2}&D_{n,2}\\
D_{n,2}&A_n+D_{n,-2}
\end{pmatrix},
\end{equation}
and
\begin{equation}\label{eq:Monsky-even}
M_{2n}=
\begin{pmatrix}
A_n^T+D_{n,2}&D_{n,-1}\\
D_{n,2}&A_n+D_{n,2}
\end{pmatrix}.
\end{equation}
Then $s_2(n)=\corank_{\F_2}M_n$ and $s_2(2n)=\corank_{\F_2}M_{2n}$; see Monsky \cite{Monsky} and
Zhang \cite[\S2]{Zhang26}.

We use the Kummer convention
$(x,y)\mapsto(x-n,x+n,x)$, with the usual extension at the rational
$2$-torsion points.  Thus Selmer classes are represented by triples
$(d_1,d_2,d_3)\in(\Q^\times/\Q^{\times2})^3$ with square product,
whose associated intersection of quadrics is
\[
-nz^2+d_2u_2^2-d_3u_3^2=0,\qquad
-nz^2+d_3u_3^2-d_1u_1^2=0.
\]
The pure group is the quotient by the Kummer images of rational
$2$-torsion.  When the Monsky matrix is zero,
we use the classes
\[
\Lambda_t=(t,1,t),\qquad \Lambda_t'=(t,t,1).
\]

We will use the following identities of Qin.

\begin{proposition}[Qin]\label{prop:qin}
Let $n$ be a positive squarefree integer.  Define $\Delta_Q(n)$ and $\sigma_Q(n)$ by
\eqref{eq:defect}.
\begin{enumerate}[label=\textup{(\roman*)}]
\item If $L(E_n,1)\ne0$, then
\[
|\Sha(E_n)|_{\mathrm{an}}
 =\left(\frac{\Delta_Q(n)}{2\sigma_Q(n)}\right)^2.
\]
If $L(E_n,1)=0$, then $\Delta_Q(n)=0$.
\item If $m>1$ is odd and squarefree, with $m\equiv1\pmod8$, then
\[
r_f(m)+r_g(m)=4h(-m),
\qquad
\frac{r_f(m)-r_g(m)}2=r_a(m)-h(-m).
\]
\item If $m$ is odd and squarefree, with $m\equiv1\pmod4$, then
\[
r_s(m)+r_t(m)=2h(-2m).
\]
\end{enumerate}
\end{proposition}

\begin{proof}
The vanishing statement in (i) follows from Qin's central value formulas
\cite[Thms.~3.20, 3.21]{Qin22}.  In the nonvanishing case, the analytic
Birch--Swinnerton-Dyer quotient is evaluated in
\cite[Cor.~2.1, Cor.~2.2]{Qin24}, giving the displayed normalization.  For even
squarefree $n$, nonvanishing can occur only in the root number $+1$ class
$n\equiv2\pmod8$, which is the even case of the cited corollary.  The
representation identities in (ii) and (iii) are the Minkowski--Siegel consequences
used in Qin's Sections~4 and~5; the difference identity in (ii) is
\cite[Lem.~3.5, Eq.~(10)]{Qin22}.
For $m=1$, the sum in (ii) is instead $r_f(1)+r_g(1)=2$; the field
$\Q(i)$ has additional units.  The difference identity in (ii) still holds
directly at $m=1$.  The families considered below have $m>1$.
\end{proof}

The Cassels pairing on $\Sel_2'(E_n)$ is alternating.  We recall the
description of its radical.

\begin{proposition}\label{prop:Cassels-nondeg}
Let $S_n=\Sel_2'(E_n)$.  The radical of its Cassels pairing is the image of
\[
\Sel_4(E_n)\longrightarrow\Sel_2(E_n)\longrightarrow S_n,
\]
where the first map is induced by multiplication by $2$ on $E_n[4]$.
Consequently the pairing is nondegenerate if and only if
\[
\rk E_n(\Q)=0,\qquad
\Sha(E_n)[2^\infty]\cong(\Z/2\Z)^{s_2(n)}.
\]
For even $s_2(n)$ this is equivalent to $\Pf(C_{\mathcal B})=1$.
\end{proposition}

\begin{proof}
The description of the kernel follows from second descent
\cite{Cassels98}; see also \cite[Lem.~4.1]{Zhang26}.
No nonzero rational $2$-torsion point of $E_n$ is twice a rational point:
the halving criterion at the roots $-n,0,n$ would require either a negative
number to be a rational square or both $n$ and $2n$ to be squares.
Thus $E_n(\Q)[4]=E_n(\Q)[2]$, and the Kummer map embeds
$E_n(\Q)[2]$ in $\Sel_2(E_n)$.

The image of $E_n(\Q)/2E_n(\Q)$ lies in the radical, since every rational
point has a Kummer lift to $\Sel_4$.  Hence a nondegenerate pure pairing forces
rank zero.  In rank zero, $S_n\cong\Sha(E_n)[2]$, and its radical corresponds
to $2\Sha(E_n)[4]$.  This is zero precisely when the $2$-primary group has
exponent at most $2$: any element of order at least $4$ would give a nonzero
element of $2\Sha[4]$.  The group is then finite because $\Sha[2]$ is finite.
The converse follows from the same description.  Finally
$\det C_{\mathcal B}=\Pf(C_{\mathcal B})^2$, which proves the last assertion.
\end{proof}

We record the specialization of Zhang's local calculation that is needed
here.  Giving the local argument also covers $Q=1,2$, for which the
auxiliary vector condition in his Lemma~3.1 is not literally applicable.

\begin{lemma}[Local pairing formulas]\label{lem:pairing-formulas}
Let $n=\varepsilon p_1\cdots p_k$, where $\varepsilon\in\{1,2\}$,
the primes $p_i$ are distinct and $1$ modulo $8$, and
$(p_i/p_j)=1$ for $i\ne j$.  Put $m_i=n/p_i$ and
\[
\Lambda_i=(p_i,1,p_i),\qquad \Lambda_i'=(p_i,p_i,1).
\]
Choose primitive positive solutions of
\[
p_i\alpha_i^2+m_i\beta_i^2=4\gamma_i^2,\qquad
p_i{\alpha_i'}^2-m_i{\beta_i'}^2=4{\gamma_i'}^2,\qquad
p_i{\alpha_i''}^2-2m_i{\beta_i''}^2=4{\gamma_i''}^2.
\]
Such solutions exist.  The additive Cassels pairing satisfies
\begin{align}
\langle\Lambda_i,\Lambda_j\rangle
 &=\left[\frac{\gamma_i'\gamma_i''}{p_j}\right],\label{eq:zh1}\\
\langle\Lambda_i,\Lambda_i'\rangle
 &=c_i+\left[\frac{\gamma_i'}{p_i}\right],\label{eq:zh2}\\
\langle\Lambda_i,\Lambda_j'\rangle
 &=\left[\frac{\gamma_i'}{p_j}\right],\label{eq:zh3}\\
\langle\Lambda_i',\Lambda_j\rangle
 &=\left[\frac{\gamma_i}{p_j}\right]
  =\left[\frac{\gamma_j'}{p_i}\right],\label{eq:zh4}\\
\langle\Lambda_i',\Lambda_j'\rangle
 &=\left[\frac{\gamma_i\gamma_i'}{p_j}\right],\label{eq:zh5}
\end{align}
where $c_i=[(1+\sqrt2)/p_i]$ and formulas involving two indices are
understood for $i\ne j$.  In addition,
\begin{equation}\label{eq:zhdiag}
\left[\frac{\gamma_i'}{p_i}\right]
 =\left[\frac{\gamma_i}{p_i}\right].
\end{equation}
\end{lemma}

\begin{proof}
The Monsky matrices vanish, so the displayed classes belong to the
Selmer group and their constituent conics have rational points.
Equivalently, local solubility of the three conics follows from the
quadratic residue assumptions at primes dividing $2n$; at other primes
they have good reduction, and they are soluble over $\mathbb R$.
The Hasse principle and a rational parametrization give positive
primitive integral solutions.

We use the tangent forms in the proof of
\cite[Prop.~3.2, Eqs.~(3.4)--(3.5)]{Zhang26}.
Writing the homogenizing coordinate as $z$, for $\Lambda_i$ these are
\[
\begin{aligned}
L_1&=m_i\beta_i'z-\alpha_i'u_2+2\gamma_i'u_3,\\
L_2&=u_3+u_1,\\
L_3&=2m_i\beta_i''z+2\gamma_i''u_1-\alpha_i''u_2.
\end{aligned}
\]
For $\Lambda_i'$ they are
\[
\begin{aligned}
L_1&=m_i\beta_i z-2\gamma_i u_2+\alpha_i u_3,\\
L_2&=m_i\beta_i'z-\alpha_i'u_3+2\gamma_i'u_1,\\
L_3&=u_1+u_2.
\end{aligned}
\]
Pairing with $\Lambda_j$ uses the Hilbert symbol of $L_1L_3$ and $p_j$;
pairing with $\Lambda_j'$ uses $L_1L_2$ and $p_j$.
The contributions at $2$ and $\infty$ vanish because every $p_j$ is a
square in $\Q_2$ and is positive.  At primes outside $2n$ they vanish
by the good reduction criterion for the local pairing
\cite[Lem.~2.1]{Zhang26}.  At $p_l\mid n$, $l\ne j$, they vanish
because $(p_j/p_l)=1$.

At $p_i$, one may use
\[
(z,u_1,u_2,u_3)=(1,\sqrt{-2m_i},0,\sqrt{-m_i})
\]
for $\Lambda_i$, and $(1,\sqrt{-m_i},\sqrt{m_i},0)$ for
$\Lambda_i'$.  These local square roots exist by the hypotheses.
At $p_j$, $j\ne i$, use $(0,1,\sqrt{p_i},1)$ and
$(0,1,1,\sqrt{p_i})$, respectively.  Signs may be chosen so that
the tangent values used are nonzero; the norm equations give the
following representatives up to local squares:
\[
\begin{array}{c|c|cc}
\text{source}&\text{product}&p_i&p_j\ (j\ne i)\\\hline
\Lambda_i&L_1L_2&(1+\sqrt2)\gamma_i'&\gamma_i'\\
\Lambda_i&L_1L_3&\sqrt2\gamma_i'\gamma_i''&\gamma_i'\gamma_i''\\
\Lambda_i'&L_1L_2&\gamma_i\gamma_i'&\gamma_i\gamma_i'\\
\Lambda_i'&L_1L_3&(1+\sqrt2)\gamma_i&\gamma_i
\end{array}
\]
This gives all five pairing formulas.  Alternation applied to
$\Lambda_i'$ gives $[\gamma_i\gamma_i'/p_i]=0$, which proves
\eqref{eq:zhdiag}; symmetry gives the reciprocal equality in
\eqref{eq:zh4}.  This computation uses only the displayed local
square conditions, and applies to both values of $\varepsilon$.
\end{proof}

The curves $E_n$ have complex multiplication by $\Z[i]$.  The following
two results relate their analytic rank to the algebraic conditions in
Proposition~\ref{prop:Cassels-nondeg}.

\begin{lemma}\label{lem:cm-bsd}
Let $n$ be a positive squarefree integer.  If $L(E_n,1)\ne0$, then the full Birch--Swinnerton-Dyer
formula holds for $E_n$, including the $2$-primary part.  In particular,
\[
|\Sha(E_n)|=|\Sha(E_n)|_{\mathrm{an}}<\infty.
\]
\end{lemma}

\begin{proof}
Apply \cite[Cor.~2]{BurungaleFlach}, which descends the theorem over a CM
field to the fixed field of complex conjugation.  Here $F=K=\Q(i)$ and
$F^+=\Q$, and all geometric endomorphisms of $E_n$ are defined over $K$.
The extension $K(E_{n,\mathrm{tors}})/K$ is abelian by complex multiplication.
The $(-1)$-twist of $E_n$ is $\Q$-isomorphic to $E_n$, so
\[
L(E_n/K,s)=L(E_n/\Q,s)^2.
\]
Thus the assumed nonvanishing over $\Q$ gives the nonvanishing required over
$K$, and the cited corollary yields BSD over $\Q$, including its $2$-part.
\end{proof}

\begin{lemma}\label{lem:rank-converse}
If $\corank_{\Z_2}\Sel_{2^\infty}(E_n/\Q)=0$, then $L(E_n,1)\ne0$.
\end{lemma}

\begin{proof}
This is the $p=2$ case of the rank zero converse theorem of Burungale and Tian
\cite[Thm.~1.1]{BurungaleTian}, which applies to every CM elliptic curve over $\Q$.
\end{proof}

\begin{proof}[Proof of Theorem~\ref{thm:main-principle}]
Assume first that $L(E_n,1)\ne0$.  By Lemma~\ref{lem:cm-bsd} and
Proposition~\ref{prop:qin},
\begin{equation}\label{eq:valuation}
\left|\frac{\Delta_Q(n)}{2\sigma_Q(n)}\right|^2=|\Sha(E_n)|.
\end{equation}
In particular $E_n(\Q)$ has rank zero, so
\[
s_2(n)=\dim_{\F_2}\Sha(E_n)[2].
\]
For a finite abelian $2$-group $G$ one has
$v_2(|G|)\ge\dim_{\F_2}G[2]$, with equality if and only if $G$ is elementary abelian.
Taking $2$-adic valuations in \eqref{eq:valuation} therefore gives
\[
v_2\!\left(\frac{\Delta_Q(n)}{2\sigma_Q(n)}\right)
 =\frac12v_2\bigl(|\Sha(E_n)|\bigr)
 \ge\frac{s_2(n)}2.
\]
Since $n$ is squarefree, $\sigma_Q(n)$ is a power of $2$; hence the valuation
inequality shows that the quotient in \eqref{eq:DQ} is integral.  Its residue modulo
$2$ is $1$ if and only if
\[
\Sha(E_n)[2^\infty]\cong(\Z/2\Z)^{s_2(n)},
\]
which, by Proposition~\ref{prop:Cassels-nondeg}, is equivalent to nondegeneracy of the
Cassels pairing and hence to $\Pf(C_{\mathcal B})=1$.

Now suppose that $L(E_n,1)=0$.  Proposition~\ref{prop:qin} gives $\Delta_Q(n)=0$, so the
left side of \eqref{eq:main-principle} is $0$.  If
$\Pf(C_{\mathcal B})=1$, then Proposition~\ref{prop:Cassels-nondeg} would imply
$\rk E_n(\Q)=0$ and finite $\Sha(E_n)[2^\infty]$.  The exact sequence for the
$2^\infty$-Selmer group would then give
$\corank_{\Z_2}\Sel_{2^\infty}(E_n/\Q)=0$, contradicting
Lemma~\ref{lem:rank-converse}.  Hence the Pfaffian is also $0$.

Finally, under a change of basis $P\in\mathrm{GL}_{s_2(n)}(\F_2)$,
\[
\Pf(P^{T}C_{\mathcal B}P)=\det(P)\Pf(C_{\mathcal B}).
\]
Since every nonzero element of $\F_2$ is $1$, $\det(P)=1$.  Thus the Pfaffian is basis
independent over $\F_2$.
\end{proof}

\begin{theorem}[Divisibility from the radical]\label{thm:radical-depth}
Let $s=s_2(n)$ be even, let $d=\dim_{\F_2}\rad C_{\mathcal B}(E_n)$, and put
$H_n=\Delta_Q(n)/(2\sigma_Q(n))$.  Then
\begin{equation}\label{eq:radical-bound}
v_2(H_n)\geq\frac{s+d}{2}=s-\frac{\rk C_{\mathcal B}(E_n)}2.
\end{equation}
Equality holds if and only if
\begin{equation}\label{eq:radical-equality}
\rk E_n(\Q)=0,\qquad
\Sha(E_n)[2^\infty]\cong
(\Z/2\Z)^{s-d}\oplus(\Z/4\Z)^d.
\end{equation}
If $H_n\ne0$, the pairing rank and $v_2(H_n)$ determine the number of cyclic
factors of order $2$, and the total excess of the remaining exponents over $2$.
\end{theorem}

\begin{proof}
If $H_n=0$, the bound is immediate.  Otherwise Qin's formula for the central value gives $L(E_n,1)\ne0$.  Lemma~\ref{lem:cm-bsd} gives rank zero, finite
$G=\Sha(E_n)[2^\infty]$, and $2v_2(H_n)=\log_2|G|$.
Write $G\cong\bigoplus_{j=1}^s\Z/2^{a_j}\Z$ with $a_j\geq1$.
By the radical description in Proposition~\ref{prop:Cassels-nondeg},
\[
d=\dim_{\F_2}2G[4]=\#\{j:a_j\geq2\}.
\]
Consequently
\begin{equation}\label{eq:excess-depth}
2v_2(H_n)=s+d+\sum_{a_j\geq2}(a_j-2).
\end{equation}
This proves the bound and its equality characterization in the nonvanishing
case.  Conversely, \eqref{eq:radical-equality} gives finite
$\Sel_{2^\infty}(E_n)$; Lemma~\ref{lem:rank-converse} then gives
$L(E_n,1)\ne0$, so the preceding argument applies.  Finally the number of factors of order $2$ is $s-d=\rk C_{\mathcal B}$, and the last sum is
$2v_2(H_n)-s-d$.
\end{proof}

\begin{corollary}[Two cases of the group structure]\label{cor:next-groups}
Retain the notation of Theorem~\ref{thm:radical-depth}.
\begin{enumerate}[label=\textup{(\roman*)}]
\item If $s\geq2$ and $v_2(H_n)=s/2+1$, then
\[
\rk E_n(\Q)=0,\qquad
\Sha(E_n)[2^\infty]\cong(\Z/2\Z)^{s-2}\oplus(\Z/4\Z)^2.
\]
\item If $d\geq2$ and $v_2(H_n)=(s+d)/2+1$, then
\[
\rk E_n(\Q)=0,\qquad
\Sha(E_n)[2^\infty]\cong
(\Z/2\Z)^{s-d}\oplus(\Z/4\Z)^{d-2}\oplus(\Z/8\Z)^2.
\]
\end{enumerate}
\end{corollary}

\begin{proof}
The nonzero defect implies analytic rank zero.  Put
$G=\Sha(E_n)[2^\infty]$ and write
$G\cong\bigoplus_j\Z/2^{a_j}\Z$ as in the proof of
Theorem~\ref{thm:radical-depth}.  The finite Cassels--Tate pairing is perfect
and alternating, so the elementary divisors of $G$ occur in equal pairs
\cite{Cassels}.  In (i),
$\sum_j(a_j-1)=2$, which allows exactly one pair with exponent $2$.
In (ii), the sum in \eqref{eq:excess-depth} is $2$, so precisely one pair of
the $d$ exponents is $3$, and the others are $2$.
\end{proof}

\begin{remark}\label{rem:s2two}
When $s_2(n)=2$, the Pfaffian in Theorem~\ref{thm:main-principle} is a single
pairing value.  When $s_2(n)=4$, it is a sum of three products of pairing values.
\end{remark}

\section{Products of two primes}\label{sec:four}

We first determine when $s_2(pq)$ or $s_2(2pq)$ equals $4$.

\begin{proposition}\label{prop:rank4-classification}
Let $p\ne q$ be odd primes and let $n\in\{pq,2pq\}$.  Then
\[
s_2(n)=4
\]
if and only if
\[
p\equiv q\equiv1\pmod8,
\qquad
\left(\frac qp\right)=1.
\]
Otherwise, $s_2(n)\le3$, and equality can occur.
\end{proposition}

\begin{proof}
The matrices in \eqref{eq:Monsky} and \eqref{eq:Monsky-even} have size $4\times4$.
Hence $s_2(n)=4$ is equivalent to the corresponding matrix being zero.

For $n=pq$, the vanishing of the off-diagonal block forces
$D_{pq,2}=0$, the upper left block then forces $A_{pq}=0$, and the lower right block
forces $D_{pq,-2}=0$.  The two diagonal conditions together are equivalent to
$p\equiv q\equiv1\pmod8$, while $A_{pq}=0$ is equivalent to
$(q/p)=(p/q)=1$.  Under $p\equiv q\equiv1\pmod4$, quadratic reciprocity identifies the
two symbols.

For $n=2pq$, vanishing of \eqref{eq:Monsky-even} is equivalent to
$D_{pq,2}=D_{pq,-1}=0$ and $A_{pq}=0$.  Thus
$p\equiv q\equiv1\pmod8$ and $(q/p)=1$.  The converse is immediate in both cases.
If the matrix is nonzero its corank is at most $3$.
\end{proof}

\begin{remark}\label{rem:s2three}
The bound in Proposition~\ref{prop:rank4-classification} is sharp.  For $p=5$ and
$q=41$ one has $(q/p)=1$ and
\[
M_{205}=
\begin{pmatrix}
1&0&1&0\\
0&0&0&0\\
1&0&1&0\\
0&0&0&0
\end{pmatrix},
\]
so $s_2(205)=3$.
\end{remark}

Under \eqref{eq:layer}, both class numbers below are divisible by $16$.

\begin{proposition}\label{prop:class-vacuous}
Assume \eqref{eq:layer}.  Then
\[
16\mid h(-pq),\qquad 16\mid h(-2pq).
\]
In particular,
\[
h(-2pq)\equiv h(-pq)\pmod{16}.
\]
Thus the class number conditions modulo $16$ in
\cite[Thm.~4.9(ii), Thm.~5.4(1)]{Qin22} hold for every pair satisfying
\eqref{eq:layer}.
\end{proposition}

\begin{proof}
Put $m=pq$.  Since $p,q\equiv1\pmod8$ and $(q/p)=1$, the matrix $A_m$ is the zero
$2\times2$ matrix.  Zhang's class group formula
\cite[Thm.~2.8(1)]{Zhang26} therefore gives
\[
h_4(-m)=h_4(-2m)=\corank A_m=2,
\]
where $h_4$ denotes the $4$-rank of the corresponding narrow class group in Zhang's
notation; for the negative discriminants here this is the ordinary class group.  Genus
theory gives $2$-rank $2$ for both discriminants
$-4pq$ and $-8pq$, since each is a product of three prime discriminants.  Thus the
$2$-primary class group has two cyclic factors, both of order at least $4$.  Its order is
therefore divisible by $4^2=16$ in both cases:
\[
16\mid h(-pq),\qquad 16\mid h(-2pq).
\]
The final congruence follows.
\end{proof}

Assume \eqref{eq:layer}.  Since $M_{pq}=0$, the pure Selmer group of
$E_{pq}$ has the ordered basis
\begin{equation}\label{eq:basis4}
\mathcal B=(\Lambda_p,\Lambda_q,\Lambda_p',\Lambda_q')
 =\bigl((p,1,p),(q,1,q),(p,p,1),(q,q,1)\bigr).
\end{equation}

Choose primitive positive solutions as in \eqref{eq:symbols-intro}.
The value of
$c_p$ is independent of the choice of $r_p$: replacing $r_p$ by $-r_p$ replaces
$r_p+1$ by $1-r_p$, and $(1+r_p)(1-r_p)\equiv-1\pmod p$ with $(-1/p)=1$.
The same argument applies to $c_q$.

\begin{proposition}\label{prop:matrix4}
With respect to the basis \eqref{eq:basis4}, the Cassels matrix is
\begin{equation}\label{eq:C4}
C_{pq}=
\begin{pmatrix}
0&\delta&v+c_p&v\\
\delta&0&u&u+c_q\\
v+c_p&u&0&u+v\\
v&u+c_q&u+v&0
\end{pmatrix}.
\end{equation}
Consequently,
\begin{equation}\label{eq:pf4}
\Pf(C_{pq})
 =c_pc_q+c_pu+c_qv+\delta(u+v)
 =\Pi(p,q).
\end{equation}
\end{proposition}

\begin{proof}
Apply Lemma~\ref{lem:pairing-formulas} with $n=pq$ and
$(p_1,p_2)=(p,q)$.  For $p_1=p$ the three norm equations are the equations
used to define $\gamma,\gamma',\gamma''$.  For $p_2=q$, the plus equation is the same
after interchanging the first two variables, so its third coordinate may be chosen to
be $\gamma$.

The entries are therefore
\begin{align*}
\langle\Lambda_p,\Lambda_q\rangle&=\delta,\\
\langle\Lambda_p,\Lambda_p'\rangle&=c_p+v,\\
\langle\Lambda_p,\Lambda_q'\rangle&=v,\\
\langle\Lambda_q,\Lambda_p'\rangle&=u,\\
\langle\Lambda_q,\Lambda_q'\rangle&=c_q+u,\\
\langle\Lambda_p',\Lambda_q'\rangle&=u+v.
\end{align*}
The reciprocal identities \eqref{eq:zh4} and the diagonal identity
\eqref{eq:zhdiag} express the coordinates from the equations with leading
coefficient $q$ in terms of $u$ and $v$.  This gives \eqref{eq:C4}.

For a $4\times4$ alternating matrix with upper triangular entries
$a,b,c,d,e,f$, the Pfaffian is $af+be+cd$.  Substituting the entries of
\eqref{eq:C4} gives \eqref{eq:pf4}.
\end{proof}

\begin{proof}[Proof of Theorem~\ref{thm:intro-four}]
By Proposition~\ref{prop:matrix4}, $\Pi(p,q)$ is the Pfaffian of the Cassels matrix.
Hence its value is independent of the auxiliary primitive solutions.  The assertions about rank and the Shafarevich--Tate group follow from
Proposition~\ref{prop:Cassels-nondeg}.
\end{proof}

\begin{proof}[Proof of Theorem~\ref{thm:intro-congruence}]
Here $\sigma_Q(pq)=\sigma_0(pq)=4$ and $s_2(pq)=4$.  Theorem
\ref{thm:main-principle} therefore gives
\[
\frac{r_f(pq)-r_g(pq)}{32}
\equiv\Pf(C_{pq})
=\Pi(p,q)\pmod2.
\]
Since $pq\equiv1\pmod8$, Proposition~\ref{prop:qin}(ii) gives
\[
\frac{r_f(pq)-r_g(pq)}2=r_a(pq)-h(-pq).
\]
Dividing by $16$ yields \eqref{eq:mod32-intro}.
\end{proof}

The pure Selmer group of $E_{2pq}$ also has dimension $4$, with the basis
\eqref{eq:basis4}.  Choose primitive positive solutions of
\begin{align*}
p\alpha_p^2+2q\beta_p^2&=4\gamma_p^2,\\
q\alpha_q^2+2p\beta_q^2&=4\gamma_q^2,\\
p{\alpha_p'}^2-2q{\beta_p'}^2&=4{\gamma_p'}^2,\\
q{\alpha_q'}^2-2p{\beta_q'}^2&=4{\gamma_q'}^2,\\
p{\alpha_p''}^2-4q{\beta_p''}^2&=4{\gamma_p''}^2.
\end{align*}
These solutions exist by Lemma~\ref{lem:pairing-formulas}.  Set
\[
u_{\mathrm e}=\left[\frac{\gamma_p}{q}\right],\qquad
v_{\mathrm e}=\left[\frac{\gamma_p'}{q}\right],\qquad
\delta_{\mathrm e}=\left[\frac{\gamma_p'\gamma_p''}{q}\right],
\]
\[
w_p=c_p+\left[\frac{\gamma_p'}p\right],
\qquad
w_q=c_q+\left[\frac{\gamma_q'}q\right],
\]
and define
\begin{equation}\label{eq:Pi2}
\Pi_2(p,q)=\delta_{\mathrm e}(u_{\mathrm e}+v_{\mathrm e})+w_p w_q+u_{\mathrm e}v_{\mathrm e}\in\F_2.
\end{equation}

\begin{theorem}\label{thm:even4}
Assume \eqref{eq:layer}.  The value $\Pi_2(p,q)$ is independent of the auxiliary
primitive solutions above.  The equality
\[
\Pi_2(p,q)=1
\]
holds if and only if
\[
2pq\text{ is noncongruent}
\quad\text{and}\quad
\Sha(E_{2pq})[2^\infty]\cong(\Z/2\Z)^4.
\]
Furthermore,
\begin{equation}\label{eq:even-cong}
\frac{r_s(pq)-r_t(pq)}{32}\equiv\Pi_2(p,q)\pmod2,
\end{equation}
and therefore
\begin{equation}\label{eq:even-cong2}
r_s(pq)\equiv h(-2pq)+16\Pi_2(p,q)\pmod{32}.
\end{equation}
\end{theorem}

\begin{proof}
Under \eqref{eq:layer}, \eqref{eq:Monsky-even} is zero, so the four classes in
\eqref{eq:basis4} form a basis.  Apply Lemma~\ref{lem:pairing-formulas} with $n=2pq$ and
$(p_1,p_2)=(p,q)$.
For $p_1=p$ the three norm equations have coefficients $2q,-2q,-4q$, as in
the definitions of $\gamma_p,\gamma_p',\gamma_p''$.  For $p_2=q$ the plus and minus
equations have coefficients $2p$ and $-2p$.  The six upper triangular entries are
therefore
\begin{align*}
\langle\Lambda_p,\Lambda_q\rangle
 &=\left[\frac{\gamma_p'\gamma_p''}{q}\right]=\delta_{\mathrm e},\\
\langle\Lambda_p,\Lambda_p'\rangle
 &=c_p+\left[\frac{\gamma_p'}p\right]=w_p,\\
\langle\Lambda_p,\Lambda_q'\rangle
 &=\left[\frac{\gamma_p'}q\right]=v_{\mathrm e},\\
\langle\Lambda_q,\Lambda_p'\rangle
 &=\left[\frac{\gamma_q'}p\right]
  =\left[\frac{\gamma_p}q\right]=u_{\mathrm e},\\
\langle\Lambda_q,\Lambda_q'\rangle
 &=c_q+\left[\frac{\gamma_q'}q\right]=w_q,\\
\langle\Lambda_p',\Lambda_q'\rangle
 &=\left[\frac{\gamma_p\gamma_p'}q\right]=u_{\mathrm e}+v_{\mathrm e}.
\end{align*}
In the fourth equality, \eqref{eq:zh4} gives $[\gamma_q'/p]=u_{\mathrm e}$.  The Pfaffian is
\[
\delta_{\mathrm e}(u_{\mathrm e}+v_{\mathrm e})+w_p w_q+u_{\mathrm e}v_{\mathrm e}=\Pi_2(p,q).
\]
Since the displayed expression is the Pfaffian of the Cassels matrix, it is independent
of the auxiliary primitive solutions.  The nondegeneracy criterion follows from
Proposition~\ref{prop:Cassels-nondeg}.

Since $\sigma_Q(2pq)=4$ and $s_2(2pq)=4$, Theorem~\ref{thm:main-principle} gives
\eqref{eq:even-cong}.  Finally Proposition~\ref{prop:qin}(iii) gives
$r_s(pq)+r_t(pq)=2h(-2pq)$.  Hence
\[
\frac{r_s(pq)-h(-2pq)}{16}=\frac{r_s(pq)-r_t(pq)}{32}
\equiv\Pi_2(p,q)\pmod2,
\]
which is equivalent to \eqref{eq:even-cong2}.
\end{proof}

We checked the formulas by enumerating the pairs in \eqref{eq:layer}.  There are $600$ pairs with $pq<250000$.  Direct enumeration of
$r_a(pq)$ and reduced binary quadratic forms shows that
$(r_a(pq)-h(-pq))/16$ is integral in every case and odd in $342$ cases; the even
quantity $(r_s(pq)-h(-2pq))/16$ is odd in $312$ cases.  For $pq<600000$ there are $1453$ pairs.  The corresponding quotients
are odd in $797$ and $740$ cases, respectively.  The R\'edei formulas of Theorem~\ref{thm:choice-free}
were checked against both defects for all $1453$ pairs.  Independently solving
the norm equations verifies both matrices for all $77$ pairs with $pq<40000$.  These computations
are not used in the proofs.

The table gives examples for $E_{pq}$.  The
norm equations in \eqref{eq:symbols-intro} were solved directly, while $r_a$ and the
class numbers were computed independently.  Put
$\rho=(r_a(pq)-h(-pq))/16\pmod2$.

\begin{center}
\small
\begin{tabular}{rrrrrrrrrrr}
\toprule
$p$&$q$&$c_p$&$c_q$&$u$&$v$&$\delta$&$\Pi$&$\rho$&$h(-pq)$&$r_a(pq)$\\
\midrule
17&89 &1&1&1&0&0&0&0&16 &16\\
17&137&1&0&1&0&1&0&0&32 &32\\
41&73 &0&1&0&1&0&1&1&48 &32\\
17&257&1&0&0&1&1&1&1&48 &32\\
41&113&0&0&1&0&0&0&0&32 &32\\
17&281&1&1&1&1&1&1&1&32 &48\\
73&89 &1&1&0&0&1&1&1&64 &80\\
73&97 &1&1&0&1&1&1&1&64 &48\\
\bottomrule
\end{tabular}
\end{center}

\section{Cassels matrices and R\'edei symbols}\label{sec:general}

We now allow any number of prime factors.  Throughout this section let
\begin{equation}\label{eq:general-hyp}
n=p_1\cdots p_k,
\qquad
p_i\equiv1\pmod8,
\qquad
\left(\frac{p_i}{p_j}\right)=1\quad(i\ne j),
\end{equation}
with $k\geq1$ and the $p_i$ distinct.  Then $A_n=0$ and $D_{n,\pm2}=0$, so $M_n=0$ and
$s_2(n)=2k$.  The natural ordered basis is
\begin{equation}\label{eq:general-basis}
\mathcal B_n=(\Lambda_1,\ldots,\Lambda_k,
\Lambda_1',\ldots,\Lambda_k'),
\qquad
\Lambda_i=(p_i,1,p_i),\quad \Lambda_i'=(p_i,p_i,1).
\end{equation}

For each $i$, choose primitive positive solutions with third coordinates
$\gamma_i,\gamma_i',\gamma_i''$ of
\begin{align}
 p_i\alpha_i^2+\frac{n}{p_i}\beta_i^2&=4\gamma_i^2,
 \label{eq:general-plus}\\
 p_i{\alpha_i'}^2-\frac{n}{p_i}{\beta_i'}^2&=4{\gamma_i'}^2,
 \label{eq:general-minus}\\
 p_i{\alpha_i''}^2-\frac{2n}{p_i}{\beta_i''}^2&=4{\gamma_i''}^2.
 \label{eq:general-minus2}
\end{align}
As before, put
\begin{equation}\label{eq:ci-general}
c_i:=\left[\frac{r_i+1}{p_i}\right],
\qquad r_i^2\equiv2\pmod{p_i}.
\end{equation}
The value of $c_i$ is independent of the choice of $r_i$.

Define $k\times k$ matrices $X,Y,Z$ over $\F_2$ as follows.  Set
$X_{ii}=Z_{ii}=0$, and for $i\ne j$ put
\begin{equation}\label{eq:block-XZ}
X_{ij}=\left[\frac{\gamma_i'\gamma_i''}{p_j}\right],
\qquad
Z_{ij}=\left[\frac{\gamma_i\gamma_i'}{p_j}\right].
\end{equation}
For the mixed block put
\begin{equation}\label{eq:block-Y}
Y_{ij}=
\begin{cases}
 c_i+\displaystyle\left[\dfrac{\gamma_i'}{p_i}\right],&i=j,\\[3mm]
 \displaystyle\left[\dfrac{\gamma_i'}{p_j}\right],&i\ne j.
\end{cases}
\end{equation}

\begin{proposition}[Cassels matrix for a product of primes]\label{prop:block-matrix}
Under \eqref{eq:general-hyp}, the Gram matrix of the Cassels pairing with
respect to \eqref{eq:general-basis} is
\begin{equation}\label{eq:block-C}
C_n=\begin{pmatrix}
X&Y\\
Y^T&Z
\end{pmatrix}.
\end{equation}
The blocks $X$ and $Z$ are alternating.  Each entry in the fixed basis is
independent of the primitive solutions chosen in
\eqref{eq:block-XZ}--\eqref{eq:block-Y}.  A change of basis replaces $C_n$
by a congruent matrix and leaves $\Pf(C_n)$ unchanged.
\end{proposition}

\begin{proof}
Apply Lemma~\ref{lem:pairing-formulas} with $\varepsilon=1$.  Formula \eqref{eq:zh1} gives the block
$X$, formulas \eqref{eq:zh2} and \eqref{eq:zh3} give $Y$, and
\eqref{eq:zh5} gives $Z$.  Formula \eqref{eq:zh4} identifies the opposite
mixed block with $Y^T$.  The blocks $X$ and $Z$ are alternating because the Cassels pairing is
alternating.
\end{proof}

We record the Pfaffian expansion for this block matrix.  If $I\subseteq\{1,\ldots,k\}$,
write $X_I$ and $Z_I$ for the principal submatrices indexed by $I$.  If
$I,J\subseteq\{1,\ldots,k\}$, write $Y_{I,J}$ for the corresponding rectangular
submatrix.  We use $\Pf(\varnothing)=\det(\varnothing)=1$.

\begin{proposition}[Pfaffian of a block matrix]\label{prop:block-pf}
Let $X,Z$ be alternating $k\times k$ matrices over $\F_2$ and let $Y$ be any
$k\times k$ matrix.  Then
\begin{equation}\label{eq:block-pf}
\Pf\begin{pmatrix}X&Y\\Y^T&Z\end{pmatrix}
=
\sum_{r=0}^{\lfloor k/2\rfloor}
\ \sum_{\substack{I,J\subseteq\{1,\ldots,k\}\\ |I|=|J|=2r}}
\Pf(X_I)\Pf(Z_J)\det\bigl(Y_{I^c,J^c}\bigr)
\quad\text{in }\F_2.
\end{equation}
\end{proposition}

\begin{proof}
Expand the Pfaffian as a sum over perfect matchings of the $2k$ vertices
corresponding to the two blocks.  Suppose that the vertices paired internally in
the first block form $I$ and those paired internally in the second block form
$J$.  If $|I|=2r$ and $|J|=2s$, then the remaining vertices can be paired across
the two blocks only when $r=s$.  For fixed $I$ and $J$, the sum over internal
matchings contributes $\Pf(X_I)\Pf(Z_J)$, while the sum over the cross matchings
is the permanent of $Y_{I^c,J^c}$.  Over $\F_2$ the signs disappear, so this
permanent is the determinant.  Summing over $r,I,J$ proves
\eqref{eq:block-pf}.
\end{proof}

For $k=2$, Proposition~\ref{prop:block-pf} reads
\begin{equation}\label{eq:k2block}
\Pf(C_n)=X_{12}Z_{12}+\det Y.
\end{equation}
In the basis used in Section~\ref{sec:four}, this is exactly
\[
\delta(u+v)+
\det\begin{pmatrix}v+c_p&v\\u&u+c_q\end{pmatrix}
=c_pc_q+c_pu+c_qv+\delta(u+v)=\Pi(p,q).
\]
For $k=3$, let
$\widehat i=\{1,2,3\}\setminus\{i\}$.  Write $X_{\widehat i}$ for the
upper off-diagonal entry of the submatrix indexed by $\widehat i$,
and define $Z_{\widehat j}$ similarly.  Then
\begin{equation}\label{eq:k3-pf}
\Pf(C_{p_1p_2p_3})
=\det Y+\sum_{i=1}^3\sum_{j=1}^3
X_{\widehat i}Z_{\widehat j}Y_{ij}.
\end{equation}
Its entries are obtained from
\eqref{eq:general-plus}--\eqref{eq:general-minus2}.

\begin{theorem}\label{thm:general}
Assume \eqref{eq:general-hyp}, and let $C_n$ be the explicit matrix
\eqref{eq:block-C}.  Then:
\begin{enumerate}[label=\textup{(\roman*)}]
\item $\Pf(C_n)=1$ if and only if
\[
\rk E_n(\Q)=0,
\qquad
\Sha(E_n)[2^\infty]\cong(\Z/2\Z)^{2k};
\]
\item one has
\begin{equation}\label{eq:general-cong}
4^k\mid r_a(n)-h(-n)
\end{equation}
and
\begin{equation}\label{eq:general-cong2}
r_a(n)\equiv h(-n)+4^k\Pf(C_n)\pmod{2\cdot4^k}.
\end{equation}
\item writing $r=\rk C_n$, one has
\begin{equation}\label{eq:general-radical}
2^{3k-r/2}\mid r_a(n)-h(-n).
\end{equation}
Equality of $2$-adic valuations is equivalent to rank zero and
\[
\Sha(E_n)[2^\infty]\cong(\Z/2\Z)^r\oplus(\Z/4\Z)^{2k-r}.
\]
\end{enumerate}
\end{theorem}

\begin{proof}
The hypotheses imply $M_n=0$, hence $s_2(n)=2k$, and the classes in
\eqref{eq:general-basis} are the standard basis of the kernel.  Part (i) follows
from Proposition~\ref{prop:Cassels-nondeg}.

For (ii), $\sigma_Q(n)=\sigma_0(n)=2^k$ and $s_2(n)/2=k$.  Theorem
\ref{thm:main-principle} gives
\[
2^{2k+1}\mid r_f(n)-r_g(n),
\qquad
\frac{r_f(n)-r_g(n)}{2^{2k+1}}
 \equiv\Pf(C_n)\pmod2.
\]
Since $n\equiv1\pmod8$, Proposition~\ref{prop:qin}(ii) gives
$r_f(n)-r_g(n)=2(r_a(n)-h(-n))$.  Division by $2$ proves
\eqref{eq:general-cong} and \eqref{eq:general-cong2}.
For (iii), $r_a(n)-h(-n)=2^kH_n$ and $d=2k-r$; now use
Theorem~\ref{thm:radical-depth}.
\end{proof}

\begin{corollary}\label{cor:k3}
Let $p<q<r$ satisfy \eqref{eq:general-hyp}.  If the right-hand side of
\eqref{eq:k3-pf} is $1$, then $pqr$ is noncongruent,
\[
\Sha(E_{pqr})[2^\infty]\cong(\Z/2\Z)^6,
\qquad
|\Sha(E_{pqr})[2^\infty]|=64,
\]
and
\[
r_a(pqr)\equiv h(-pqr)+64\pmod{128}.
\]
\end{corollary}

\begin{remark}
The order $64$ in Corollary~\ref{cor:k3} refers to the $2$-primary part.
The Pfaffian does not determine the primary components at odd primes.
\end{remark}

For $k=3$ there are $28$ triples in this family with $pqr<2\,500\,000$.  The
quantity $(r_a(pqr)-h(-pqr))/64$ is odd in $11$ cases.  Solving all nine
norm equations for each triple independently verifies the complete matrices,
\eqref{eq:k3-pf}, and \eqref{eq:general-cong2} in all $28$ cases.  These computations are not used
in the proof.

We use Stevenhagen's additive R\'edei symbol
$[a,b,c]_{\mathrm R}\in\F_2$.  Its admissibility conditions are
\[
(a,b)_v=(a,c)_v=(b,c)_v=1\quad\text{for every place }v,
\]
\[
\gcd(\Delta(a),\Delta(b),\Delta(c))=1,
\]
where $\Delta(a)=\operatorname{disc}\Q(\sqrt a)$.
For admissible triples it is symmetric and linear in each argument
\cite[Def.~7.8, Thm.~7.9]{Stevenhagen}; a trivial squareclass gives zero.
All R\'edei symbols below satisfy these conditions, since the primes in
\eqref{eq:general-hyp} are $1$ modulo $8$ and are pairwise quadratic residues.
When $(a/\ell)=1$ and $\ell\equiv1\pmod4$, write
\[
\left[\frac a\ell\right]_4\in\F_2,
\qquad
(-1)^{[a/\ell]_4}\equiv a^{(\ell-1)/4}\pmod\ell.
\]

The next lemma relates the norm coordinates to R\'edei symbols, including
at primes dividing the coefficients of the equation.

\begin{lemma}[Norm coordinates and R\'edei symbols]\label{lem:norm-redei}
Assume \eqref{eq:general-hyp}, fix $p\mid n$, and put $m=n/p$.
Let $b\in\{m,-m,2m,-2m\}$, and suppose
\[
p\alpha^2+b\beta^2=4\gamma^2
\]
has a primitive positive integral solution.  For every prime $q\mid n$,
\begin{equation}\label{eq:gamma-redei}
\left[\frac\gamma q\right]=[p,b,q]_{\mathrm R}.
\end{equation}
The assertion also holds for $m=1$.  If $b=1$, the right-hand side is zero.
\end{lemma}

\begin{proof}
First, $\gamma$ is prime to $n$.  Indeed, if a prime divisor of $p$ or $m$
divided $\gamma$, the equation and squarefreeness would force it to divide
both $\alpha$ and $\beta$, contrary to primitivity.
Put $K_p=\Q(\sqrt p)$ and
\[
\eta=2\gamma+\alpha\sqrt p,
\qquad N_{K_p/\Q}(\eta)=b\beta^2.
\]
For nontrivial $b$, the field
$F=\Q(\sqrt p,\sqrt b,\sqrt\eta)$ is a R\'edei extension for $(p,b)$.
We check its ramification before evaluating its Artin symbol.

At an odd prime outside $pb$, a common prime divisor of $\eta$ and its
conjugate would divide both $\alpha$ and $\gamma$, and hence $\beta$.
Primitivity excludes this.  The norm identity then gives even valuations
of $\eta$, so the quadratic extension is unramified there.  At an odd
prime dividing $b$, $p$ is a local square.  The two valuations of $\eta$
in $K_p$ are zero and odd; after adjoining $\sqrt b$ the latter becomes
even.  The resulting extension over $\Q(\sqrt{pb})$ is unramified.
At $p$, $\eta$ is a unit and the same tame criterion applies.  Thus only
dyadic ramification remains.  Since $p\equiv1\pmod8$, there exists
$t\in\{1,-1,2,-2\}$ such that
\[
F_t=\Q(\sqrt p,\sqrt b,\sqrt{t\eta})
\]
is minimally ramified
\cite[Props.~7.2--7.3, Cor.~7.4]{Stevenhagen}.

If $q\mid m$, choose the prime of $K_p$ above $q$ at which $\eta$ is a
unit.  The residue fields do not change on passing to
$\Q(\sqrt p,\sqrt b)$: its quadratic extension at this prime is ramified.  In its residue field $\eta\equiv4\gamma\pmod q$.
If $q=p$, the residue is $2\gamma\pmod p$, and $b$ is a local
square at $p$.  Thus this prime also has residue field $\F_p$.
The Frobenius on the quadratic extension defined by $t\eta$ is therefore
the quadratic character of $\gamma$: both $t$ and $2$ are squares modulo
every prime dividing $n$.  This proves \eqref{eq:gamma-redei}, using the
local description of the R\'edei symbol
\cite[Eq.~(47)]{Stevenhagen}.  When $m=b=1$, parity gives
$\alpha=2A$, $\beta=2B$, and $\gamma$ odd, with
$pA^2+B^2=\gamma^2$ and $\gcd(B,\gamma)=1$.
If $B$ is even, the coprime factors $\gamma-B$ and $\gamma+B$
are a square and $p$ times a square, in some order; hence $\gamma$
is a square divided by $2$ modulo $p$.
If $B$ is odd, apply the same argument to the coprime factors
$(\gamma-B)/2$ and $(\gamma+B)/2$, whose product is $p(A/2)^2$.
Their sum is $\gamma$, which is a square modulo $p$.
Since $(2/p)=1$, both cases give $[\gamma/p]=0$.
\end{proof}

\begin{proposition}[The dyadic character]\label{prop:ci-redei}
Let $\ell\equiv1\pmod8$ be prime and put
$e_\ell=(\ell-1)/8\pmod2$.  Then
\begin{equation}\label{eq:ci-redei}
c_\ell=[2,-1,\ell]_{\mathrm R}=[-1,\ell,2]_{\mathrm R},
\qquad
c_\ell+\left[\frac2\ell\right]_4=e_\ell.
\end{equation}
Moreover,
\begin{equation}\label{eq:ci-class}
c_\ell\equiv h(-\ell)/4\pmod2.
\end{equation}
\end{proposition}

\begin{proof}
The norm identity $N_{\Q(\sqrt2)/\Q}(1+\sqrt2)=-1$ gives the R\'edei
extension for $(2,-1)$.  Its generator is a unit, so it has no odd
ramification.  Its quadratic character at $\ell$ is $c_\ell$; the
admissible dyadic twists have trivial character at $\ell$.
R\'edei symmetry proves the first equality.

The real cyclotomic element
$\zeta_{16}+\zeta_{16}^{-1}=\sqrt{2+\sqrt2}$ changes sign under Frobenius
at $\ell\equiv9\pmod{16}$ and is fixed at $\ell\equiv1\pmod{16}$.
Consequently $[(2+\sqrt2)/\ell]=e_\ell$.  Since
$2+\sqrt2=\sqrt2(1+\sqrt2)$ and
$[\sqrt2/\ell]=[2/\ell]_4$, the second identity follows.
Finally \cite{BarrucandCohn} gives
$c_\ell=0$ if and only if $8\mid h(-\ell)$.
The $4$-rank is one by \cite[Thm.~2.8(1)]{Zhang26}, so
$4\mid h(-\ell)$ and \eqref{eq:ci-class} follows.
\end{proof}

For $i\ne j$ define
\[
b_{ij}=\left[\frac{p_j}{p_i}\right]_4,
\quad R_{ij}=[2,p_i,p_j]_{\mathrm R},
\quad T_{ij}=\sum_{\ell\ne i,j}[p_i,p_j,p_\ell]_{\mathrm R}.
\]
An empty sum is zero.  Put
$D_2=\operatorname{diag}([2/p_1]_4,\ldots,[2/p_k]_4)$.

\begin{theorem}[R\'edei block matrices]\label{thm:redei-block}
Under \eqref{eq:general-hyp}, the blocks in \eqref{eq:block-C} satisfy
\begin{equation}\label{eq:redei-blocks}
\begin{aligned}
X_{ij}&=R_{ij},& Z_{ij}&=b_{ij}+b_{ji} &&(i\ne j),\\
Y_{ij}&=b_{ij}+T_{ij} &&&&(i\ne j),\\
Y_{ii}&=c_i+\sum_{j\ne i}b_{ij},&X_{ii}&=Z_{ii}=0.
\end{aligned}
\end{equation}
In the same ordered basis, the Cassels matrix for $E_{2n}$ is
\begin{equation}\label{eq:even-block-redei}
C_{2n}=\begin{pmatrix}
X&Y+X+D_2\\
Y^T+X+D_2&Z
\end{pmatrix}.
\end{equation}
Thus these matrices are expressed entirely through rational quartic and
R\'edei symbols.
\end{theorem}

\begin{proof}
Write $m_i=n/p_i$.  Lemma~\ref{lem:norm-redei} applied to the three norm
equations gives, for $i\ne j$,
\begin{align*}
X_{ij}&=[p_i,-m_i,p_j]_{\mathrm R}+[p_i,-2m_i,p_j]_{\mathrm R}
       =[2,p_i,p_j]_{\mathrm R},\\
Y_{ij}&=[p_i,-m_i,p_j]_{\mathrm R},\\
Z_{ij}&=[p_i,m_i,p_j]_{\mathrm R}+[p_i,-m_i,p_j]_{\mathrm R}
       =[-1,p_i,p_j]_{\mathrm R}.
\end{align*}
For the prime arguments used here, repeated symbols satisfy
$[p,p,c]_{\mathrm R}=[c/p]_4$
\cite[Def.~7.1]{Stevenhagen}.  Also
$[p_i,p_j,-p_ip_j]_{\mathrm R}=0$
\cite[Prop.~7.10]{Stevenhagen}.  Hence
\[
[-1,p_i,p_j]_{\mathrm R}=b_{ij}+b_{ji}.
\]
Expanding $-m_i=-1\cdot p_j\prod_{\ell\ne i,j}p_\ell$ gives
$Y_{ij}=b_{ij}+T_{ij}$.  On the diagonal, symmetry gives
\[
[p_i,-m_i,p_i]_{\mathrm R}
=[p_i,p_i,-m_i]_{\mathrm R}
=\left[\frac{-m_i}{p_i}\right]_4
=\sum_{j\ne i}b_{ij},
\]
since $[-1/p_i]_4=0$.

For $E_{2n}$, the relevant coefficients are $2m_i,-2m_i,-4m_i$.
In the third equation, replacing $2\beta_i''$ by a new coordinate
reduces $-4m_i$ to $-m_i$; primitive normalization changes the third
coordinate only by a power of $2$, invisible in all the quadratic symbols.
Linearity then gives the same blocks $X$ and $Z$.  Off the diagonal the
mixed block gains $[p_i,2,p_j]_{\mathrm R}=X_{ij}$, while on the diagonal
it gains $[p_i,2,p_i]_{\mathrm R}=[2/p_i]_4$.
This proves \eqref{eq:even-block-redei}.
\end{proof}

\begin{corollary}[The family $E_{2n}$]\label{cor:even-general}
Assume \eqref{eq:general-hyp} and let $r=\rk C_{2n}$.  Then
\[
r_s(n)\equiv h(-2n)+4^k\Pf(C_{2n})\pmod{2\cdot4^k},
\qquad
2^{3k-r/2}\mid r_s(n)-h(-2n).
\]
The Pfaffian is $1$ if and only if $E_{2n}$ has rank zero and
$\Sha(E_{2n})[2^\infty]\cong(\Z/2\Z)^{2k}$.
Equality in the valuation bound holds precisely in rank zero with
\[
\Sha(E_{2n})[2^\infty]\cong(\Z/2\Z)^r\oplus(\Z/4\Z)^{2k-r}.
\]
\end{corollary}

\begin{proof}
Here $s_2(2n)=2k$, $\sigma_Q(2n)=2^k$, and
$r_s(n)-h(-2n)=(r_s(n)-r_t(n))/2=2^kH_{2n}$.
Apply Theorems~\ref{thm:main-principle} and \ref{thm:radical-depth}.
\end{proof}

\begin{theorem}[Formulas for $n=pq$]\label{thm:choice-free}
Assume \eqref{eq:layer}, and set
\[
U=\left[\frac pq\right]_4,\quad V=\left[\frac qp\right]_4,
\quad R=[2,p,q]_{\mathrm R},
\]
\[
e_p=\frac{p-1}{8}\pmod2,\qquad e_q=\frac{q-1}{8}\pmod2.
\]
Then $u=U$, $v=V$, $\delta=R$, and
\begin{align}
\Pi(p,q)&=c_pc_q+c_pU+c_qV+R(U+V),\label{eq:Pi-redei}\\
\Pi_2(p,q)&=R+e_pe_q+e_pU+e_qV.\label{eq:Pi2-redei}
\end{align}
Both formulas are independent of auxiliary solutions of norm equations.
\end{theorem}

\begin{proof}
For $k=2$ the sums $T_{ij}$ vanish, and \eqref{eq:redei-blocks} gives
\[
X=\begin{pmatrix}0&R\\R&0\end{pmatrix},\quad
Y=\begin{pmatrix}c_p+V&V\\U&c_q+U\end{pmatrix},\quad
Z=\begin{pmatrix}0&U+V\\U+V&0\end{pmatrix}.
\]
Comparison with \eqref{eq:C4} proves the identifications and
\eqref{eq:Pi-redei}.  By \eqref{eq:ci-redei}, the even mixed block is
\[
Y+X+D_2=\begin{pmatrix}e_p+V&V+R\\U+R&e_q+U\end{pmatrix}.
\]
Its Pfaffian is $R(U+V)+(e_p+V)(e_q+U)+(V+R)(U+R)$,
which simplifies to \eqref{eq:Pi2-redei} in $\F_2$.
\end{proof}

The symbol $R$ can be computed from a single representation
$p=a^2-2b^2$.  If $r_q^2\equiv2\pmod q$, then
\begin{equation}\label{eq:redei-computation}
R=\left[\frac{a+br_q}{q}\right].
\end{equation}
Indeed $a+b\sqrt2$ has norm $p$ and yields the minimally ramified
R\'edei extension after a dyadic twist, which is invisible at $q$.
The two choices of $r_q$ give the same value because $(p/q)=1$.
The existence of this norm representation follows from the Euclidean
ring $\Z[\sqrt2]$ and its unit of norm $-1$.

\begin{proposition}[Quartic data are insufficient]\label{prop:quartic-insufficient}
There is no function $F:\F_2^4\to\F_2$ such that, for every pair satisfying
\eqref{eq:layer},
\[
\Pi(p,q)=F\!\left(c_p,c_q,
 \left[\frac pq\right]_4,\left[\frac qp\right]_4\right).
\]
The same assertion holds with $\Pi_2(p,q)$ in place of $\Pi(p,q)$.
\end{proposition}

\begin{proof}
The following examples suffice.  The last column is computed from
\eqref{eq:mod32-intro} and \eqref{eq:even-cong2}, using direct enumeration
of representations and reduced positive binary quadratic forms.
\[
\begin{array}{c|cc|cccc|c}
& p&q&c_p&c_q&[p/q]_4&[q/p]_4&\text{Pfaffian}\\ \hline
\Pi&17&257&1&0&0&1&1\\
\Pi&17&457&1&0&0&1&0\\ \hline
\Pi_2&41&113&0&0&1&0&1\\
\Pi_2&41&353&0&0&1&0&0
\end{array}
\]
For the first two rows,
\[
\begin{aligned}
(r_a(17\cdot257),h(-17\cdot257))&=(32,48),\\
(r_a(17\cdot457),h(-17\cdot457))&=(96,96).
\end{aligned}
\]
so the normalized odd defects have opposite parity.  For the last two rows,
\[
(r_s(41\cdot113),h(-2\cdot41\cdot113))=(128,144),
\]
\[
(r_s(41\cdot353),h(-2\cdot41\cdot353))=(144,144),
\]
so the normalized even defects again have opposite parity.  In each pair of rows, the four dyadic and quartic symbols agree.  Hence no such function
$F$ can exist.
\end{proof}

The values of $R$ in the four rows are $1,0,0,1$.  Thus the R\'edei
symbol distinguishes the rows in each pair.

To compute the joint densities, fix a prime $p\equiv1\pmod8$ and write
\[
\mathcal P_p=\{q\ne p:q\text{ is prime},\ q\equiv1\pmod8,\ (q/p)=1\}.
\]
Here density means natural density within the primes:
\[
\delta(\mathcal S)=\lim_{X\to\infty}
\frac{\#\{q\le X:q\in\mathcal S\}}{\pi(X)}.
\]
The set $\mathcal P_p$ has density $1/8$.  We first determine the field
intersections needed to distribute the five symbols
$e_q,c_q,U,V,R$.

Choose $a,b>0$ with $p=a^2-2b^2$, and put
\[
\epsilon=1+\sqrt2,\qquad \alpha=a+b\sqrt2,\qquad
F_p=\Q(i,\sqrt2,\sqrt p).
\]
Let $T_p$ be the unique quartic subfield of $\Q(\zeta_p)$.  It is cyclic
over $\Q$ and contains $\Q(\sqrt p)$.  Define
\begin{equation}\label{eq:density-fields}
\begin{aligned}
N_p&=F_p(\sqrt\epsilon,\sqrt[4]p,\sqrt\alpha),&
A_p&=\Q(\zeta_{16})T_p,\\
M_p&=A_pN_p,&B_p&=\Q(\zeta_{16p}).
\end{aligned}
\end{equation}

\begin{lemma}[Field intersections and conjugacy classes]\label{lem:density-fields}
These fields are Galois over $\Q$, and
\begin{equation}\label{eq:density-intersections}
[N_p:F_p]=8,\qquad N_p\cap A_p=N_p\cap B_p=F_p.
\end{equation}
In particular,
\begin{equation}\label{eq:density-degrees}
[M_p:\Q]=256,\qquad [B_pN_p:\Q]=64(p-1).
\end{equation}
For $G_p=\operatorname{Gal}(M_p/\Q)$, its center is
\[
H_p=\operatorname{Gal}(M_p/F_p)\cong(\Z/2\Z)^5,
\]
and its commutator subgroup has order $8$.  There are $32$ conjugacy
classes of size $1$ and $56$ conjugacy classes of size $4$.
For $\operatorname{Gal}(B_pN_p/\Q)$, the corresponding numbers are
$8(p-1)$ and $14(p-1)$.
\end{lemma}

\begin{proof}
Let $\tau_0,\tau_2,\tau_p$ be the automorphisms of $F_p$ that change the
sign of $i,\sqrt2,\sqrt p$, respectively, and fix the other two generators.
For $r,s,t\in\{0,1\}$ put
\[
\beta=\epsilon^r(\sqrt p)^s\alpha^t,\qquad z=\sqrt\beta.
\]
The identities
\[
\tau_2(\epsilon)=-\epsilon^{-1},\qquad
\tau_2(\alpha)=p\alpha^{-1},\qquad
\tau_p(\sqrt p)=-\sqrt p
\]
show that $F_p(z)$ is Galois over $\Q$.  On $z$, lifts of the three
automorphisms act, up to signs, by
\[
\widetilde\tau_0(z)=z,\qquad
\widetilde\tau_2(z)=
\frac{i^r(\sqrt p)^t}{\epsilon^r\alpha^t}\,z,\qquad
\widetilde\tau_p(z)=i^s z.
\]
Their commutators therefore act as follows; the signs are unaffected by
the choices of lifts:
\begin{equation}\label{eq:density-commutators}
\begin{array}{c|ccc}
&[\widetilde\tau_0,\widetilde\tau_2]
&[\widetilde\tau_0,\widetilde\tau_p]
&[\widetilde\tau_2,\widetilde\tau_p]\\\hline
z&(-1)^r z&(-1)^s z&(-1)^t z.
\end{array}
\end{equation}
For example, $\tau_0$ changes the sign of $i^r$ in the second lift,
whereas $\tau_p$ changes the sign of $(\sqrt p)^t$.
If $\beta$ were a square in $F_p$, then $z\in F_p$ and the commuting
automorphisms $\tau_0,\tau_2,\tau_p$ would act on it by the same
formulas, up to signs.  The three commutator signs would have to be $1$,
forcing $r=s=t=0$.  Hence no nontrivial product of
$\epsilon,\sqrt p,\alpha$ is a square in $F_p$.  For a nontrivial
product, \eqref{eq:density-commutators} also shows that
$F_p(z)/\Q$ is nonabelian.  Kummer theory gives $[N_p:F_p]=8$.

The sign changes on the three square roots form a central subgroup
$\operatorname{Gal}(N_p/F_p)$.  Applied to each root separately,
\eqref{eq:density-commutators} shows that this subgroup is generated by
the three commutators.  It is consequently the full commutator subgroup
of $\operatorname{Gal}(N_p/\Q)$.  Hence $F_p$ is the largest subfield
of $N_p$ that is abelian over $\Q$.  Both $A_p$ and $B_p$ are abelian
over $\Q$ and contain $F_p$, proving \eqref{eq:density-intersections}.
This also proves that the three quadratic extensions over $B_p$
defined by the displayed radicals are linearly disjoint.

The fields $\Q(\zeta_{16})$ and $T_p$ are linearly disjoint, since
they lie in cyclotomic fields of coprime conductors.  Thus
$[A_p:\Q]=8\cdot4=32$, while $[B_p:\Q]=8(p-1)$.
Since $[F_p:\Q]=8$ and $[N_p:\Q]=64$, the intersection formula
gives \eqref{eq:density-degrees}.

The restriction map identifies $H_p$ with
\[
\operatorname{Gal}(A_p/F_p)\times\operatorname{Gal}(N_p/F_p)
\cong(\Z/2\Z)^2\times(\Z/2\Z)^3.
\]
It is central in $G_p$.  The three independent commutators in
\eqref{eq:density-commutators} identify the commutator map with
\[
\bigwedge\nolimits^2\F_2^3\longrightarrow\F_2^3,
\qquad x\wedge y\longmapsto [\widetilde x,\widetilde y],
\]
an isomorphism in the indicated bases.  For every nonzero $x$, the
map $y\mapsto x\wedge y$ has rank $2$.  An element outside $H_p$
therefore has a conjugacy class of size $4$, whereas every element of
$H_p$ is central.  This proves the description of the center and gives
$(256-32)/4=56$ remaining classes.
For $B_pN_p$, the same commutator map gives center
$\operatorname{Gal}(B_pN_p/F_p)$ of order $8(p-1)$ and
$14(p-1)$ remaining classes, each of size $4$.
\end{proof}

\begin{theorem}[Joint densities]\label{thm:fixed-prime-density}
As $q$ varies over $\mathcal P_p$, the vector
\[
(e_q,c_q,U,V,R)\in\F_2^5
\]
has the uniform limiting distribution.  Each of its $32$ values occurs
on a set of primes of density $1/256$.
For $s,t\in\F_2$, put
\[
\mathcal P_p^{s,t}=\{q\in\mathcal P_p:
\Pi(p,q)=s,\ \Pi_2(p,q)=t\}.
\]
Then
\begin{equation}\label{eq:joint-density}
\delta(\mathcal P_p^{s,t})
=\frac{4+(-1)^{s+c_p}
 +\mathbf1_{\{c_p=e_p\}}(-1)^{s+t}}{128}.
\end{equation}
Equivalently, the relative densities within $\mathcal P_p$ are
\[
\begin{array}{cc|cccc}
c_p&e_p&(0,0)&(0,1)&(1,0)&(1,1)\\\hline
0&0&3/8&1/4&1/8&1/4\\
0&1&5/16&5/16&3/16&3/16\\
1&0&3/16&3/16&5/16&5/16\\
1&1&1/4&1/8&1/4&3/8
\end{array}
\]
All four pairs occur for every fixed $p\equiv1\pmod8$.
\end{theorem}

\begin{proof}
The field $M_p$ is unramified outside $2p$.  Admissibility of $q$ means
that it splits completely in $F_p$, so its Frobenius lies in $H_p$.
The two coordinates on $\operatorname{Gal}(A_p/F_p)$ are
$e_q$, determined by $q$ modulo $16$, and $V$, determined by the image
of $q$ in the cyclic quartic quotient of $(\Z/p\Z)^\times$.
The three coordinates on $\operatorname{Gal}(N_p/F_p)$ are the signs
on $\sqrt\epsilon,\sqrt[4]p,\sqrt\alpha$, namely $c_q,U,R$.
Here the last identification follows from \eqref{eq:redei-computation}.
Lemma~\ref{lem:density-fields} proves that these five coordinates are
independent and that each element of $H_p$ is a singleton conjugacy
class.  Chebotarev's theorem \cite{Neukirch} assigns density $1/256$
to each class.  Their union has density $32/256=1/8$.

It remains to count the classes giving each pair of Pfaffians.
Write $P=\Pi(p,q)$, $Q=\Pi_2(p,q)$, and let $\mathbb E$ denote
the average over the $32$ symbol vectors.  Formula
\eqref{eq:Pi2-redei} gives $\mathbb E(-1)^Q=0$, by averaging over $R$.
In \eqref{eq:Pi-redei}, averaging over $R$ leaves only $U=V$;
averaging over $c_q$ then leaves $U=V=c_p$.  Consequently
\[
\mathbb E(-1)^P=\frac{(-1)^{c_p}}4.
\]
For $P+Q$, averaging over $R$ leaves $U+V=1$.  The coefficients of
$c_q$ and $e_q$ are then $c_p+V$ and $e_p+V$, respectively.
Averaging over these two variables leaves $V=c_p=e_p$, and the
remaining exponent is zero.  Thus
\[
\mathbb E(-1)^{P+Q}=\frac{\mathbf1_{\{c_p=e_p\}}}{4}.
\]
The identity
\[
\mathbf1_{\{P=s,Q=t\}}
=\frac14(1+(-1)^{P+s})(1+(-1)^{Q+t})
\]
now gives the stated relative densities.  More explicitly, the number
of singleton Frobenius classes for the pair $(s,t)$ is
\[
2\bigl(4+(-1)^{s+c_p}
 +\mathbf1_{\{c_p=e_p\}}(-1)^{s+t}\bigr).
\]
Dividing by $256$ proves \eqref{eq:joint-density}.
\end{proof}

\begin{corollary}[Independence and simultaneous noncongruence]\label{cor:density-patterns}
Within $\mathcal P_p$, the limiting proportions for
$\Pi(p,q)=1$ and $\Pi_2(p,q)=1$ are, respectively,
\[
\frac{4-(-1)^{c_p}}8\qquad\text{and}\qquad\frac12.
\]
The two Pfaffians are independent in their limiting distribution if
and only if $2$ is not a fourth power modulo $p$.
For $p=17$, the four prime densities, in the order
$(0,0),(0,1),(1,0),(1,1)$, are
\begin{equation}\label{eq:p17-densities}
\frac3{128},\qquad\frac3{128},\qquad
\frac5{128},\qquad\frac5{128}.
\end{equation}
In particular, the set of primes $q\in\mathcal P_{17}$ for which
\[
\rk E_{17q}(\Q)=\rk E_{34q}(\Q)=0,
\]
\[
\Sha(E_{17q})[2^\infty]\cong\Sha(E_{34q})[2^\infty]
\cong(\Z/2\Z)^4
\]
has density $5/128$.
\end{corollary}

\begin{proof}
Summing over the other Pfaffian gives the marginal proportions.
The joint distribution is their product exactly when $c_p\ne e_p$.
By \eqref{eq:ci-redei}, this is equivalent to $[2/p]_4=1$.
For $p=17$ one has $c_p=1$ and $e_p=0$, which gives
\eqref{eq:p17-densities}.  The final assertion follows from
Theorems~\ref{thm:intro-four} and \ref{thm:even4}.
\end{proof}

The last density concerns the specified $2$-primary groups.  It gives
a lower bound for the lower prime density of the larger set on which
$17q$ and $34q$ are both noncongruent; the density of that larger set
is not determined here.

Theorems~\ref{thm:radical-depth} and \ref{thm:redei-block}, together with
Corollary~\ref{cor:next-groups}, determine the groups in the table below.
The representation numbers and class numbers were computed by independent
enumeration.  Put $D=r_a(pq)-h(-pq)$, so that $H_{pq}=D/4$.
\begin{center}
\begin{tabular}{rrrrl}
\toprule
$p$&$q$&$\rk C_{pq}$&$D$&$\Sha(E_{pq})[2^\infty]$\\
\midrule
17&1889&2&$-32$&$(\Z/2\Z)^2\oplus(\Z/4\Z)^2$\\
41&5273&0&$64$&$(\Z/4\Z)^4$\\
17&3041&2&$64$&$(\Z/2\Z)^2\oplus(\Z/8\Z)^2$\\
\bottomrule
\end{tabular}
\end{center}
All three curves have rank zero because their defects are nonzero.
Here $H_{pq}^2$ is a power of $2$, so the displayed groups are in fact the
full Shafarevich--Tate groups.
The last two groups both have order $256$.  The different ranks of their
Cassels pairings distinguish their group structures.

As a consistency check, we carried out independent finite computations for
all $1453$ pairs with $pq<600000$ and all $28$ triples with
$pqr<2500000$.  These checks verify the representation congruences, the
R\'edei matrices, the radical bounds, and the examples above.  For the $77$
pairs with $pq<40000$ and all $28$ triples, the relevant Cassels matrices
were also recovered independently from primitive norm solutions.  We also
checked the $32$ symbol vectors occurring in the degree $256$ governing field
calculation.  None of these finite computations is used in the proofs.

For a nonzero defect, \eqref{eq:excess-depth} determines the sum of the
exponents beyond the contribution detected by $\Sel_4$.  In general, this
sum and the pairing rank do not determine all elementary divisors.
Further descent is needed to distinguish them.

\appendix
\renewcommand{\thetheorem}{\Alph{section}.\arabic{theorem}}
\section{The families \texorpdfstring{$15p$ and $30p$}{15p and 30p}}\label{app:fifteen}

The primes $3$ and $5$ lie outside the hypotheses of
Section~\ref{sec:general}.  Nevertheless, the general results of
Section~\ref{sec:framework} apply to these families.

\begin{proposition}[Pure Selmer dimensions]\label{prop:fifteen-selmer}
Let $p\nmid30$ be prime.  In the following table the columns are ordered
by the two multiplicative Legendre symbols $((3/p),(5/p))$:
\[
\begin{array}{c|c|cccc}
n&p\bmod8 &(+,+)&(+,-)&(-,+)&(-,-)\\\hline
15p&5&2&0&0&0\\
15p&7&2&0&2&0\\
30p&3&2&0&0&0\\
30p&7&2&0&0&2
\end{array}
\]
Every entry is the exact value of $s_2(n)$.  Consequently,
\begin{enumerate}[label=\textup{(\roman*)}]
\item if $p\equiv5\pmod8$ and $15p$ is congruent, then
$(3/p)=(5/p)=1$, equivalently $p\equiv61,109\pmod{120}$;
\item if $p\equiv7\pmod8$ and $15p$ is congruent, then
$(5/p)=1$, equivalently $p\equiv31,39\pmod{40}$;
\item if $p\equiv3\pmod8$ and $30p$ is congruent, then
$(3/p)=(5/p)=1$, equivalently $p\equiv11,59\pmod{120}$;
\item if $p\equiv7\pmod8$ and $30p$ is congruent, then
$(15/p)=1$, equivalently $(3/p)=(5/p)$.
\end{enumerate}
These are necessary conditions; an entry $2$ does not by itself imply
positive Mordell--Weil rank.
\end{proposition}

\begin{proof}
Put $a=[3/p]$, $b=[5/p]$, $h=[-1/p]$, and $t=[2/p]$.
For the ordered odd prime factors $(3,5,p)$, quadratic reciprocity gives
\[
A=\begin{pmatrix}
1+a+h&1&a+h\\
1&1+b&b\\
a&b&a+b
\end{pmatrix}.
\]
The diagonal matrices are
\[
D_{15p,2}=\operatorname{diag}(1,1,t),\qquad
D_{15p,-1}=\operatorname{diag}(1,0,h).
\]
Substitution into \eqref{eq:Monsky} and \eqref{eq:Monsky-even}
reduces the calculation to these four pairs $(h,t)$ and matrices:
\[
\begin{array}{c|c|c}
n&p\bmod8&(h,t)\\\hline
15p&5&(0,1)\\
15p&7&(1,0)\\
30p&3&(1,1)\\
30p&7&(1,0)
\end{array}
\]
For completeness, row reduction over $\F_2$ gives the following
kernel dimensions, for $a,b\in\{0,1\}$:
\[
\begin{array}{c|c}
(n,p\bmod8)&\dim\ker M_n\\\hline
(15p,5)&2(1-a)(1-b)\\
(15p,7)&2(1-b)\\
(30p,3)&2(1-a)(1-b)\\
(30p,7)&2\mathbf1_{\{a=b\}}
\end{array}
\]
The expressions in the right column are ordinary integers.  This
proves the table by checking the four substitutions for $a,b$ in each
displayed matrix.  Since $\rk E_n(\Q)\le s_2(n)$, every zero entry
forces noncongruence.  The residue classes in (i)--(iii) follow from
quadratic reciprocity and the Chinese remainder theorem.
\end{proof}

The condition that the two signs agree in (iv) cannot be replaced by requiring both
symbols to be $+1$: for $p=7$ both are $-1$, whereas the triangle
with side lengths $20,21,29$ has area $210$.

\begin{proposition}[Class numbers on the permitted $30p$ layer]
\label{prop:fifteen-class}
Let $p\nmid30$ be prime, $p\equiv7\pmod8$, and $(15/p)=1$.  Then
\begin{equation}\label{eq:fifteen-class}
h(-15p)\equiv h(-30p)\equiv
\begin{cases}
0\pmod{16},&(3/p)=(5/p)=1,\\
8\pmod{16},&(3/p)=(5/p)=-1.
\end{cases}
\end{equation}
In particular, $h(-30p)\equiv h(-15p)\pmod{16}$ throughout this
entire Selmer layer, irrespective of the rank of $E_{30p}$.
\end{proposition}

\begin{proof}
Put $a=[3/p]=[5/p]\in\F_2$.  The fundamental discriminants are
$-60p=(-4)(-3)5(-p)$ and $-120p=(-8)(-3)5(-p)$.
Both class groups therefore have $2$-rank $3$.
Use the transpose of the R\'edei matrix convention in
\cite[Thm.~3.1, Eq.~(16)]{Stevenhagen}: if the prime discriminants
are $d_0,d_1,d_2,d_3$ and $\ell_i\mid d_i$ is prime, its
off-diagonal entry is $[d_j/\ell_i]$, with diagonal chosen so that
each row sums to zero.  At $\ell_i=2$ the symbol is Kronecker.
In the orders displayed above, the two matrices are
\[
R_{-60p}=\begin{pmatrix}
0&1&1&0\\
1&a&1&a\\
0&1&1+a&a\\
1&1+a&a&0
\end{pmatrix},\qquad
R_{-120p}=\begin{pmatrix}
0&1&1&0\\
0&1+a&1&a\\
1&1&a&a\\
1&1+a&a&0
\end{pmatrix}.
\]
For $a=0$, the rows of each matrix span a space of dimension $2$.
For $a=1$, each upper left $3\times3$ minor has determinant $1$,
while the row sums are zero, so each rank is $3$.  Thus the
$4$-ranks of both class groups are $1$ and $0$, respectively.
If $a=0$, the $2$-primary group has three cyclic factors, at least
one of order $4$ or greater, so its order is divisible by $16$.
If $a=1$, all three factors have order $2$, so the class number
is $8$ times an odd integer.  This proves \eqref{eq:fifteen-class}.
\end{proof}

\begin{corollary}[Normalized defects for $15p$ and $30p$]
\label{cor:fifteen-defect}
Let $(n,p\bmod8)$ occur in Proposition~\ref{prop:fifteen-selmer}.
Since $\sigma_Q(n)=8$, put
\[
H_{15p}=\frac{r_f(15p)-r_g(15p)}{16},\qquad
H_{30p}=\frac{r_s(15p)-r_t(15p)}{16}.
\]
On a layer with entry $0$, $H_n$ is odd, $E_n$ has rank zero, and
$\Sha(E_n)[2^\infty]=0$.  On a layer with entry $2$, choose a basis
of $\Sel_2'(E_n)$ and write its Cassels matrix as
$\left(\begin{smallmatrix}0&\kappa_n\\\kappa_n&0\end{smallmatrix}\right)$.
Then
\[
\frac{H_n}{2}\equiv\kappa_n\pmod2.
\]
In particular, $\kappa_n=1$ is equivalent to rank zero with
$\Sha(E_n)[2^\infty]\cong(\Z/2\Z)^2$.  If $\kappa_n=0$, then
$4\mid H_n$; this case may require further descent.
Whenever $H_n\ne0$, the full Shafarevich--Tate group is finite and
$|\Sha(E_n)|=H_n^2$.
\end{corollary}

\begin{proof}
Apply Theorems~\ref{thm:main-principle} and
\ref{thm:radical-depth}, including the convention for the empty
Pfaffian when $s_2(n)=0$.  The last assertion follows from
Proposition~\ref{prop:qin} and Lemma~\ref{lem:cm-bsd}.
\end{proof}

\section*{Declarations}

\medskip
\noindent{\large\bfseries Funding}\par
\medskip
This research was supported by the National Natural Science Foundation of China (NSFC), Grant Nos. 12231009 and 11971224.

\bigskip
\noindent{\large\bfseries Conflict of interest}\par
\medskip
The author declares that there is no conflict of interest.

\bigskip
\noindent{\large\bfseries Data availability}\par
\medskip
Data sharing is not applicable to this article. All arguments and calculations used in the proofs are contained in the text. The finite computations reported as consistency checks are not used in the proofs.

\end{document}